\documentclass[a4paper,reqno,11pt]{amsart}

\newfont{\cyr}{wncyr10 scaled 1100}

\usepackage[left=2.7cm,right=2.7cm,top=3.5cm,bottom=3cm]{geometry}

\usepackage{amsthm,amssymb,amsmath,amsfonts,mathrsfs,amscd}
\usepackage{mathtools}
\usepackage{extarrows}
\usepackage[latin1]{inputenc}
\usepackage[all]{xy}
\usepackage{latexsym}
\usepackage{longtable}
\usepackage{xcolor}
\usepackage{comment}
\usepackage[shortlabels]{enumitem}
\usepackage{setspace}
\setdisplayskipstretch{}
\usepackage{multirow}
\usepackage{xcolor,colortbl}
\usepackage{changepage}
\usepackage{float}
\usepackage{tikz-cd}
\usepackage[colorlinks=true,
            linkcolor=black,
            citecolor=black,
            urlcolor=black]{hyperref}
\usepackage[backend=biber,style=alphabetic,maxnames=99,minnames=99, sorting=nyt]{biblatex}

\DeclareFieldFormat[article]{title}{#1}
\DefineBibliographyStrings{english}{
  in = {}
}
\AtBeginBibliography{\footnotesize}

\theoremstyle{plain}
\newtheorem{theorem}{Theorem}[section]
\newtheorem{corollary}[theorem]{Corollary}
\newtheorem{lemma}[theorem]{Lemma}
\newtheorem{proposition}[theorem]{Proposition}

\theoremstyle{definition}
\newtheorem{definition}[theorem]{Definition}

\newtheorem{examplewr}[theorem]{Example}

\theoremstyle{remark}
\newtheorem{obswr}[theorem]{Observation}
\newtheorem{remarkwr}[theorem]{Remark}

\definecolor{Gray}{gray}{0.85}
\definecolor{LightCyan}{rgb}{0.88,1,1}

\newcolumntype{g}{>{\columncolor{Gray}}c}
\newcolumntype{y}{>{\columncolor{LightCyan}}c}
\newcolumntype{o}{>{\columncolor{pink}}c}

\newenvironment{remark}{\begin{remarkwr}\begin{upshape}}{\end{upshape}\end{remarkwr}}

\newcommand{\Ch}{{\mathrm{Ch}}}
\newcommand{\Sh}{{\mathrm{Sh}}}
\newcommand{\Sm}{{\mathrm{Sm}}}
\newcommand{\CH}{{\mathrm{CH}}}
\renewcommand{\H}{{\mathrm{H}}}

\newcommand{\Z}{\mathbb{Z}}
\newcommand{\Q}{\mathbb{Q}}

\newcommand{\A}{\mathbb{A}}

\newcommand{\G}{\mathbb{G}}
\newcommand{\T}{\mathbb{T}}

\newcommand{\DM}{\mathrm{DM}}
\newcommand{\SH}{\mathrm{SH}}

\newcommand{\GL}{\mathrm{GL}}

\newcommand{\SL}{\mathrm{SL}}

\newcommand{\Hom}{\mathrm{Hom}}

\newcommand{\Pic}{\mathrm{Pic}}

\makeatletter
\newsavebox{\@brx}
\newcommand{\llangle}[1][]{\savebox{\@brx}{\(\m@th{#1\langle}\)}%
  \mathopen{\copy\@brx\kern-0.5\wd\@brx\usebox{\@brx}}}
\newcommand{\rrangle}[1][]{\savebox{\@brx}{\(\m@th{#1\rangle}\)}%
  \mathclose{\copy\@brx\kern-0.5\wd\@brx\usebox{\@brx}}}
\makeatother

\begin{document}

\title[Localization and motivic relations]{Localization and elliptic motivic relations}
\author{Peter Xu}

\begin{abstract}
We observe that the motivic analogue of Suslin reciprocity (and similar degree-zero statements) is a formal consequence of localization (plus purity/some six functor formalism). In particular, the statement of Suslin reciprocity for smooth schemes over fields due to Kriz is a corollary of localization for higher Chow groups, over any base; we write down the framework yielding such relations with coefficients for schemes smooth over any base in $\A^1$-invariant motivic cohomology. As an application, we refine some relations between cup products of modular units to be integral in coefficients and in the base: first, we imitate the (rational-coefficients, complex-analytic) Busuioc--Park--Patashnick--Stevens argument for full-level-$N$ elliptic schemes, extending the result to integral bases and coefficients using the elementary reciprocity statement. We then refine the construction and resulting relations to the setting of motivic sheaves; in particular, this gives analogous relations at non-full level structure, as well as over any smooth global quotient stack.
\end{abstract}

\maketitle
\tableofcontents

\section{Introduction and elementary reciprocity}

Let $C$ be a smooth proper curve over a field $k$ with function field $K$. The classical \emph{Weil reciprocity} law states that the total degree of the divisor of a rational function, summing up multiplicities of poles and zeroes, is zero. This has the following generalization to Milnor $K$-theory: for each codimension-$1$ point $x$ of $C$, write $\partial_x: K_n^M(K) \to K_{n-1}^M(k(x))$ for the associated tame symbol map in Milnor $K$-theory, introduced in \cite{Mil}, which may be explicitly calculated by decomposing products into elements of the form
\begin{equation} \label{eq:cup}
\{\pi, u_1, \ldots, u_{n-1} \} \mapsto \{\overline{u_1},\ldots, \overline{u_{n-1}}\}
\end{equation}
for a uniformizer $\pi$ in the valuation ring at $x$ and units $u_1,\ldots, u_{n-1}$; here, the bar denotes reduction to the residue field $k(x)$. Then \emph{Suslin reciprocity}, due to Bass--Tate \cite{BaTa} and Suslin \cite{Sus}, is the statement that the sum over all tame symbols on $K_n^M(K)$, transferred by finite pushforward down to $k$,
\begin{equation} \label{eq:suslin}
\varphi\mapsto \sum_{x\in C^{(1)}} \mathrm{N}_{k(x)/k}\partial_x \varphi,
\end{equation}
(noting this sum only has finitely many nonzero terms for any given element) is the zero map. For $n=1$, this recovers Weil reciprocity. 

Milnor $K$-theory is the ``simplest part'' of the more general theory of motivic cohomology, a fundamental geometric/arithmetic invariant which may be defined for schemes more general than spectra of fields. More precisely, the bigraded motivic cohomology theory on schemes $X$ (whose definition(s) we will examine later)
\[
X\mapsto H^p(X, \Z(q))
\]
coincides with the $n$th Milnor $K$-theory group in the regime $p=q=n$ for $X$ the spectrum of a field. Sophie Kriz \cite{Kr} formulated and proved the following generalization for the motivic cohomology of smooth schemes over a perfect field $k$.

\begin{theorem} \label{thm:kriz}
    Let $S$ be a smooth variety over a perfect field $k$, $f:C\to S$ be smooth projective of relative dimension $1$, and let $T\subset C$ be a reduced closed subscheme of codimension $1$ whose components are all dominant over $S$. Then the composite
    \[
    H^p(C-T, \Z(q)) \xrightarrow{\partial} H^{p+1}_T(C, \Z(q)) \xrightarrow{(f|_T)_*} H^{p-1}(S, \Z(q-1))
    \]
    is zero, where the two maps are given, respectively, by the localization sequence in higher Chow groups, and pushforward of cycles.
\end{theorem}

Kriz's proof interplays the Morel--Voevodsky presheaves-with-transfers model for triangulated mixed motives over a field with Bloch's classical theory of higher Chow groups. Our first observation is that, in fact, the existing theory of higher Chow groups alone affords a proof of Theorem \ref{thm:kriz} even over more general bases and with \emph{no} smoothness hypothesis.

We quickly recall the higher Chow group theory: let $X$ be \emph{any} finite-type scheme over a regular base of dimension at most $1$. Then after Bloch \cite{B} (in the field case) and Levine \cite{L2}, the Zariski sheaves of higher cycle complexes on $X$ are
\[
    \Z(n)^{BL}_{X} := \bigl(U\mapsto z^n(U\times \Delta^\bullet)\bigr)
\]
where $z^n(U\times \Delta^i)$ is the module of codimension-$n$ cycles on the algebraic $i$-simplex $\Delta^i$ over $U$ which intersect all simplicial faces properly, with the natural face differentials. The higher Chow groups of $X$ are then defined by
\[
\CH^p(X, q) := \mathbb{H}_q(X, U\mapsto z^p(U\times \Delta^\bullet))
\]
and inherit proper pushforward and flat pullback functoriality from the complexes. When $X$ is defined over a field, the hypercohomology spectral sequence collapses immediately and the higher Chow groups may be computed simply from the complex of global sections $z^p(X\times \Delta^\bullet)$; in general, for the same reason, the higher Chow groups may always be computed in hypercohomology after pushforward to the base rather than on all of $X$. Note that none of this requires any smoothness hypotheses, but the definition only generally gives the ``correct'' definition of motivic cohomology when $X$ is smooth over $B$, in the sense of natural identifications $\CH^q(X, 2q-p)\cong H^p(X,\Z(q))$.\footnote{Morally, Bloch--Levine's higher Chow groups are the \emph{Borel--Moore homology} theory for motivic spaces.}

In this generality, Levine constructed a localization triangle \cite[Theorem 1.7]{L2}: for $i:Z\to X$ an equidimensional closed immersion of codimension $d$, with $j:U\to X$ be the complementary open immersion, and writing $\pi$ for the projection map down to $B$ of all three schemes, this is a distinguished triangle in the derived category
\[
\pi_*\Z(n-d)^{BL}_Z[-2d] \xrightarrow{i_*} \pi_*\Z(n)_X^{BL} \xrightarrow{j^*} \pi_*\Z(n)_U^{BL}.
\]
The localization sequence has as immediate corollary the following strengthening of Theorem \ref{thm:kriz}, which comes from assuming $f$ of relative dimension $1$ and $C,S$ smooth over a field in the statement below.

\begin{theorem} \label{thm:naive}
    Let $f:C\to S$ be a proper map of schemes of relative dimension $d$ over a Dedekind scheme $B$, and let $T\subset C$ be a closed subscheme of $C$ of codimension $d$ finite over $S$. The composition
    \[
        \CH^q(C-T, 2q-p)\xrightarrow{\partial} \CH^{q-d}(T,2q-p-1) \xrightarrow{(f|_T)_*} \CH^{q-d}(S, 2q-p-1)
    \]
    is zero, where the maps come from the connecting map in the localization sequence for $T\hookrightarrow C$, respectively the pushforward of the restriction of $f$. 
\end{theorem}
\begin{proof}
    The map $(f|_T)_*$ factors through the pushforward by the inclusion $T\hookrightarrow C$, so this follows from the composite of two consecutive arrows in the localization triangle being zero.
\end{proof}

From this perspective, Suslin reciprocity, or more generally ``sum of residues is zero,'' is a completely formal consequence of localization sequences. For applications, the interest often comes from the case when $T$ is the disjoint union of some set of components $(T_i)_{i\in I}$, in which case we obtain relations analogous to \eqref{eq:suslin} in motivic cohomology coming from the vanishing of sums of operators
\[
\sum_{i\in I} (f|_{T_i})_* \circ \partial_{T_i}=0,
\]
whose \emph{individual} terms can be interesting non-vanishing classes in the motivic cohomology of $S$. For instance, this is the use case in the article  \cite{BPPS} which motivated us to write this note, which proves certain ``Manin'' relations between cup products of Siegel units (including the Beilinson--Kato elements whose organization into an Euler system has yielded celebrated results on modular $L$-values) using classical Suslin reciprocity. Their results are with rational coefficients and at infinite level over a complex analytic base, which makes their applicability to arithmetic somewhat difficult.

The major application of the present article is to refine their results to be integral and at arbitrary finite level structure. The former issue may be resolved only with the elementary reciprocity relation for higher Chow groups, while for the latter we find it convenient to work in the formalism of motivic sheaves.

\subsection{Summary of contents}

In the remainder of this short article, we first record a slightly more sophisticated formulation of the above observation in the setting of $\A^1$-invariant motivic cohomology with coefficients, and over slightly more general bases. 

\begin{theorem}[Theorem \ref{thm:main}]
    Let $S$ be a global quotient stack of qcqs scheme, and let $f:C\to S$ be proper smooth of relative dimension $d$, and $i:T\hookrightarrow C$ be a codimension-$d$ closed immersion of smooth schemes over $S$ with open complement $j:U\hookrightarrow C$. Let $E\in \DM(S)$ be any coefficient module for motivic cohomology on $S$. Then the composite in the $\A^1$-invariant motivic cohomology of \cite{BEM}
    \[
    H^p(U,j_*f^*E(q))\to H^{p-2d+1}(T,i^* f^*E(q-d))\to  H^{p-2d+1}(S,E(q-d))
    \]
    is the zero map, where the two maps come from a localization sequence of motivic sheaves, respectively the trace map for the finite morphism $f\circ i:T\to S$.
\end{theorem}

When the base is \emph{smooth}, higher Chow groups coincide with motivic cohomology (see discussion at end of next section), and this result may be identified with our elementary formulation Theorem \ref{thm:naive}. Outside of this regime, it is not clear what significance the higher Chow groups have.

At full level, the elementary reciprocity we have already seen is enough to refine the main result of \cite{BPPS} to motivic cohomology groups integral in the coefficients and (smooth) base, which we record in Theorem \ref{thm:bpps}. This correspondingly produces modular symbols valued in cup products of Siegel units via the Bykovskii presentation (Theorem \ref{thm:modnaive}). 

To obtain more intrinsic statements not dependent on assuming full level $N$, we frame things in terms of motivic sheaves. After suitably categorifying all constructions (which all largely hinge on the sheaf version of the theta function construction Theorem \ref{thm:motivic-integral-theta}) and transposing Busuioc--Park--Patashnick--Stevens' correction of Goncharov's summation argument, we obtain the following motivic sheaf level analogues of our main theorems, whose statements we give a little impressionistically here for lack of technical background:

\begin{theorem}[Theorem \ref{thm:sheafrel}]
    Let $\mathcal{M}$ be a smooth global quotient stack on which $N$ is invertible, and $E$ an elliptic scheme over it. Then there is a homotopy class morphism of motivic sheaves on $\mathcal{M}$
    \[
    \Psi: \text{motive of }E^3[N] \to \G_{m,\mathcal{M}}^{\otimes 2}
    \]
    which is null-homotopic after multiplication by a certain constant $N^3 d_N$. (Here, the tensor power is from the monoidal structure on motivic sheaves on $\mathcal{M}$, realized as smash product over the Eilenberg--MacLane spectrum.) The specializations of this vanishing when $E[N]$ is a constant group scheme yield the Manin relations of \cite{BPPS}.
\end{theorem}
And its modular symbol consequence:

\begin{theorem}[Theorem \ref{thm:mod}]
    There exists a $\GL_n(\Z)$-equivariant modular symbol
    \[
    \mathrm{St}(\Q^n) \to (\text{homotopy classes of maps from motive of }E[N]^n\text{ to } \G_{m,\mathcal{M}}^{\otimes n})
    \]
    sending the standard apartment class $[e_1,\ldots, e_n]$ to $N^3 d_n\mathfrak{g}^{\otimes n}$, where $\mathfrak{g}$ is a ``Siegel map'' 
    \[
    \mathfrak{g}:\text{motive of }E[N] \to \G_{m,\mathcal{M}}
    \]
    which specializes under certain natural evaluations to classical constructions of Siegel units.
\end{theorem}

Both of these motivic sheaf level statements, under suitable specialization, imply identities between motivic cohomology elements at coarser-than-full-level-$N$ which cannot be deduced directly from the full level $N$ statements without losing integral information (since \'{e}tale descent for motivic cohomology requires inverting the degree of the map). The arguments to prove these refinements largely rely on exploiting the six functor formalism to reduce to some ``universal'' case, where the naive manipulations directly work just as in the full-level $N$ setting.

\subsection{Acknowledgements} Thanks to Sophie Kriz for responding to emails and expressing interest, Cecilia Busuioc for a motivating conversation on applications, and Tess Bouis for some helpful guidance on motives.

\section{Localization in $\SH$ and motivic relations}

In this section, we assemble the necessary formalism from the literature to obtain a general ``Suslin reciprocity'' type of statement (Theorem \ref{thm:main}) in the setting of $\A^1$-invariant motivic cohomology with coefficients, and then use Spitzweck's model to check that the maps involved coincide with the elementary ones in the introduction in the absolutely smooth case (i.e. the regime in which higher Chow groups are expected to coincide with motivic cohomology). We note that for the example application in the following section, we actually only need the trivial-coefficients setting, but we believe the more general formulation may be useful in understanding relations between motivic polylogarithms.

Let $X$ be a separated Noetherian scheme over $B$, or more generally a global quotient stack of such a scheme by a fixed algebraic group (as in the introduction of \cite{Hoy}). We write $\Sm_{X,t}$ for the category of smooth schemes over $X$ endowed with a Grothendieck topology $t$, and $\Sh(\Sm_{X,t}, \mathrm{Ani})$ the $\infty$-category of sheaves of anima on the associated site. Let $\T_X := \A^1 / \G_m$ be the Tate object in $\Sh(\Sm_{X,Nis}, \mathrm{Ani})$, defined as the cofiber of the indicated inclusion of represented sheaves.
\begin{definition}
We have the following categories associated to $X$:
\begin{itemize}
    \item The \emph{unstable motivic homotopy category} $\H(X)$ is the localization of the homotopy category of $\Sh(\Sm_{X,Nis}, \mathrm{Ani})$ along all projection maps $S\times \A^1 \to S$ (or rather the associated sheaves) for any $S\in \Sm_{X,Nis}$. 
    \item The \emph{stable homotopy category} $\SH(X):= \H_*(X)[\T_X^{-1}]$ is obtained by freely smash-inverting the Tate object in the \emph{pointed} motivic homotopy category (i.e. the undercategory of the terminal object).
\end{itemize}
\end{definition}
One defines an oriented $E_\infty$-ring spectrum $H\Z_X$ in the stable homotopy category $\SH(X)$ via the slice filtration on $K\GL$ \cite{BEM}, and one then defines the stable $\infty$-category, respectively triangulated category, of motives over $X$ to be the module category and its homotopy category:
\[
h\DM(X):= H\Z_X\text{-}\mathrm{Mod},\;\;\;\DM(X) := \mathrm{Ho}(H\Z_X\text{-}\mathrm{Mod}).
\]
These are endowed with symmetric monoidal structures via the smash product of $H\Z_X$-modules. We set
\[
    H^i(X,\Z(n)) := \pi_0\Hom_{\SH(X)}(1, \Sigma^{i,n} H\Z_X) \cong \pi_0\Hom_{\DM(X)}(\Z(0)_X,\Z(n)_X[i]),
\]
the latter isomorphism being the free/forgetful adjunction between the free $H\Z_X$-module functor $\SH(X) \to \DM(X), \mathcal{S}\mapsto H\Z_X \wedge \mathcal{S}$ and the forgetful inclusion functor $\DM(X) \to \SH(X)$. Here $\Sigma^{i,n}$ denotes the bigraded suspension functor in $\SH(X)$ coming from smash product with $(S^1)^{\wedge i} \wedge \T^{\wedge n}$, and $\Z(k)_X$ is the Tate twisted coefficient module
\[
    \Z(k)_X := H\Z_X \wedge \T^{\wedge k}
\]
where $H\Z_X$ acts on the left smash factor, so $\Sigma^{i,n} H\Z_X$ and $\Z(n)_X[i]$ are just two names for the same object (identified via the forgetful functor) in our chosen nomenclature in their respective categories. More generally, given $E\in \DM(X)$, we may set 
\[
H^i(X, E(n)) := \pi_0\Hom_{\SH(X)}(1, \Sigma^{i,n}E) \cong \Hom_{\DM(X)} (\Z(0)_X, E(n)[i]).
\]
The base-change compatibility property $f^* H\Z_Y \cong  H\Z_X$ for $f:X\to Y$ implies that the motivic cohomology of $X$ may always equally be computed on the base as $H^i(Y, f_*E(n))$, as one expects.

We appeal to \cite{Hoy} for the six-functor formalism in $\SH(X)$: for a separated morphism of finite type $f:X\to Y$, there are adjoint pairs of functors
    \[
    f^*:\SH(Y) \rightleftarrows \SH(X): f_*,\quad f_!: \SH(X) \rightleftarrows \SH(Y): f^!
    \]
satisfying a series of compatibilities, including that $f_*=f_!$ for proper $f$. The following are our needed formal properties:

\begin{proposition}
Let $S$ be a qcqs scheme or global quotient stack of such a scheme, and write
\[
X\longmapsto \SH(X)
\]
for the stable motivic homotopy category. The following properties hold (e.g. as recorded in \cite[Theorem 6.18]{Hoy}, or \cite[Theorem 3.2.23]{MV} for purity of a closed pair).

\begin{enumerate}
    \item \textup{(localization)}
    Let $i:Z\hookrightarrow X$ be a closed immersion with open complement $j:U\hookrightarrow X$. Then for every $E\in \SH(X)$ there is a functorial localization triangle
    \[
    i_*i^!E\longrightarrow E\longrightarrow j_*j^*E
    \xrightarrow{\partial_Z}
    i_*i^!E[1].
    \]
    \item \textup{(relative purity)}
    If $f:X\to Y$ is smooth, then there is a functorial purity equivalence $f^!E\simeq f^*E\wedge \mathrm{Th}(T_f)$ where $T_f$ is the relative tangent bundle.

    \item \textup{(purity of a smooth closed pair)}
    If $i:Z\hookrightarrow X$ is a closed immersion of smooth schemes over a base $S$, let $N_{Z/X}$ be the normal bundle. Then there is a functorial purity equivalence
    \[
    i^!E\simeq i^*E\wedge \mathrm{Th}(-N_{Z/X}).
    \]
    \end{enumerate}
\end{proposition}
The same formal properties are then inherited for the categories of modules over the base-change compatible assignment of oriented $E_\infty$-ring spectra $S\mapsto H\Z_S$ and in the res homotopy category. In this setting, for $f$ smooth of relative dimension $d$, relative purity becomes
    \[
    f^!E\simeq f^*E(d)[2d].
    \]
and if $i$ has codimension $c$, purity of a smooth closed pair becomes
    \[
    i^!E\simeq i^*E(-c)[-2c]
    \]
by identifying the respective Thom spectra with powers of the Tate sphere. From these formal properties, it follows in the same way as usual:
\begin{theorem}\label{thm:main}
    Let $S$ be a qcqs scheme or global quotient stack thereof as before, and let $f:C\to S$ be proper smooth of relative dimension $d$, and $i:T\hookrightarrow C$ be a codimension-$d$ closed immersion of smooth schemes over $S$ with open complement $j:U:=C-T\hookrightarrow C$. Write $g= f\circ i$. Let $E\in \DM(S)$ be any $H\Z_S$-module spectrum. Then the composite in $\A^1$-invariant motivic cohomology
    \[
    H^p(U,j_*f^*E(q))\to H^{p-2d+1}(T,g^*E(q-d))\to  H^{p-2d+1}(S,E(q-d))
    \]
    is the zero map, where the two maps come from the localization sequence, respectively the trace map for the finite morphism $g:T\to S$ via the composite in $\DM(S)$
    \[
     g_*g^* E(q-d)=g_!g^!E(q-d) \to E(q-d)
    \]
    coming from the counit of the adjunction, where here we have used relative dimension-$0$ purity and the identification $g_*=g_!$ for $g$ proper.
\end{theorem}
\begin{proof}
    From localization and the two purities, we have a triangle in $\DM(C)$
    \[
    i_*g^*E(-d)[-2d] \to f^* E \to j_* j^* f^* E \to i_*g^*E(-d)[-2d+1]
    \]
    to which we may apply $\Hom_{\DM(C)}( \Z(0)_C,- \otimes \Z(q)_C)$ to obtain the standard localization sequence with supports, including the segment 
    \[
    \ldots \to H^p(U,j_*f^*E(q))\to H^{p-2d+1}(T,g^*E(q-d))\to  H^{p+1}(C,E(q)) \to \ldots
    \]
    Writing the closed immersion map on motivic cohomology from the localization triangle on $S$ as
    \[
        g_*g^* E(-d)[-2d] \to f_*f^*E
    \]
    we see that the trace map for $g$ factors through this morphism via the counit for the proper morphism $f$ (with purity twist, as in the theorem statement but in dimension $d$ rather than $0$), so we are done as before.
\end{proof}
    
We now describe the relationship of this result to the Theorems of the introduction, which is essentially that they coincide for smooth schemes over Dedekind bases: as mentioned, when $E=H\Z_S$, and all schemes are absolutely smooth, this construction agrees with the Bloch--Levine construction from the introduction: Spitzweck constructs \cite[\S4.4]{Sp} an oriented $E_\infty$-ring spectrum for $X$ smooth over a Dedekind domain
\[
M\Z_X \in \mathrm{Sp}_\T(\Ch(\Sh(\Sm_{X,Zar}, \Z))),
\]
for which one has an isomorphism of oriented $E_\infty$-ring spectra $M\Z_X \cong H\Z_X$ \cite{BH} compatible with base change. For the purposes of calculation, Spitzweck constructs also a ``naive spectrum'' $\mathcal{M}_X\in \SH(X)$ and a functorial isomorphism
\[
\mu_X:\mathcal{M}_X\xrightarrow{\sim} M\Z_X.
\]
The spectrum $\mathcal{M}_X$, not being defined via a rectified symmetric $\T$-spectrum, cannot see the $E_\infty$ structure, but has the advantage of being an $\Omega_\T$-spectrum in the sense that for all $r$, there are canonical isomorphisms
\[
M_X(r-1)[-1] \cong \underline{\Hom}(\T_X, M_X(r))
\]
inside $\SH(X)$, where on the right-hand side we have the mapping spectrum \cite[Proposition 5.29]{Sp}, and $M_X(r)[r]\in D(\Sh(\Sm_{X,Zar},\Z))$ is the ``level r'' part of $\mathcal{M}_X$ (i.e. the space $E_r$ in the tower-of-spaces description of spectra).\footnote{Spitzweck does not distinguish notation for the spectra $\mathcal{M}_X$ and the level parts we denote $M_X$; we find it helpful for our understanding to do so.}\footnote{Here, one applies the Dold--Kan functor and Nisnevich sheafifies to view these objects inside $\SH(X)$; Spitzweck frequently implicitly views complexes of big site Zariski sheaves inside the category of motivic spaces in this way and we follow his lead.} This allows Spitzweck to calculate the comparison:

\begin{proposition} \label{prop:comparison}
    Let $X$ be smooth over a Dedekind domain $B$, writing $f:X\to Y:=\mathrm{Spec }\,B$ for the structure map. The comparison $\mu$ gives rise to functorial isomorphisms
    \[
        \pi_0\Hom_{\SH(X)}(1, \Sigma^{i,n}M\Z_X) \cong \CH^{n}(X, 2n-i)
    \]
\end{proposition}
\begin{proof}
    This is left as an exercise by Spitzweck, so we give a few details. If $f:X\to Y$ is a smooth map, then the pullback of sheaves induces a functor $f^*:\SH(Y)\to \SH(X)$, which has a \emph{left} adjoint $f_\#:\SH(X) \to \SH(Y)$ defined on the level of sheaves by left Kan extension along the obvious forgetful functor $\Sm_X\to \Sm_Y$. It follows then almost tautologically that the image of the unit $f_{\#}1_{\SH(X)}$ may be identified with the suspension spectrum $\Sigma^\infty X_+\in \SH(Y)$: i.e., the spectrum given in level $n$ by $\Z[X]_{Zar} \wedge \T_Y^{\wedge n}$ with the tautological structure maps. Then we may compute:
    \begin{align}
        \Hom_{\SH(X)}(1, \Sigma^{i,n}M\Z_X) &\cong \Hom_{\SH(Y)}(\Sigma^\infty X_+, \Sigma^{i,n}M\Z_Y) 
        \\&\cong \Hom_{\SH(Y)}(\Sigma^\infty X_+, \Sigma^{i,n} \mathcal{M}_Y(n)[i]) \\&\cong \Hom_{D(\Sh(\Sm_{Y,Zar},\Z))}( \Z[X_+]_{Zar},  M_Y(n)[i])\\& \cong \mathbb{H}^i_{Nis}(X,M_Y(n)|_{X_{Nis}}) \\&\cong \mathbb{H}^i_{Zar}(X,\Z(n)_X^{BL})
    \end{align}
    which is the higher Chow group. Here, the identification (3) is the adjunction between $f_\#$ and $f^*$ together with base change of $M\Z$, (4) is the identification $\mu_X$, (5) is the $\Omega_\T$-spectrum property of $\mathcal{M}$, (6) is the derived Yoneda lemma, and (7) is \cite[Corollary 5.20]{Sp} together with the calculation of Geisser \cite[Proposition 3.6]{G} that Nisnevich and Zariski hypercohomologies of the Bloch cycle complexes compute the same groups.
\end{proof}

Spitzweck records (\cite[Theorem 3.1]{Sp}) that Levine's localization sequences assemble into a localization sequence for the spectra $\mathcal{M}_X$, whose maps one may check essentially tautologically are identified under $\mu$ with the localization sequence for $M\Z_X$.

\section{Bykovskii relations among Beilinson--Kato elements}

\subsection{Elementary formulation}

The main result of \cite{BPPS} yields certain relations between cup products in Milnor $K$-theory of modular units, following an idea of Goncharov.\footnote{This followed earlier work of Brunault \cite{Bru} showing such relations at the level of \'{e}tale cohomology; related results were also proven in \cite{FK} and \cite{SV}.} These results were stated with $\Q$-coefficients and ``at the generic point'' on complex-analytic models of elliptic curves; for arithmetic applications, e.g. to the Sharifi conjectures (cf. \cite{SV}, \cite{LSSW}), one prefers integrality both in the coefficients and the base. Using the machine for producing motivic relations discussed in the previous sections, we prove such a refinement. 

Let $\pi:E\to S$ be any elliptic scheme, and let $N>2$ be an integer such that $E[N]$ is a constant group scheme over $S$; write $U=E-E[N]$ for the complement.\footnote{We treat here only the case where $S$ is regular; if this fails, one can still obtain analogous results by replacing all unit groups with the groups of units on semi-normalizations as a consequence of cdh-sheafification from the formalism of the previous section.}

The following is an algebraic argument substituting for the complex-analytic construction of the theta elements ${}_N\Theta_{\mathbf{a}}(u,\tau)^{12}/{}_N\Theta_{\mathbf{x}}(u,\tau)^{12}$ in loc. cit.: 

\begin{proposition}\label{prop:integral-theta}
Let $D$ be a degree-zero divisor supported on $E[N]$, disjoint from zero. Write 
\[
d_N:=\left(\gcd_{a\equiv 1\pmod{N}} a^2-1\right)
\]
Then there is a
distinguished element $\theta_{D}\in H^1(U,\Z(1)) =\mathcal{O}(U)^\times$ with divisor $d_N\cdot N\cdot D$ characterized by its invariance under $[a]_*$ for every integer $a\equiv 1\pmod{N}$. 
\end{proposition}
\begin{proof}
The unit-divisor exact sequence (or the localization sequence more generally) gives
\[
\mathcal{O}(S)^\times \cong \mathcal{O}(E)^\times \to  \mathcal{O}(E-E[N])^{\times} \to \Z\{E[N]\} \to \Pic(E)(S)
\]
The image of the class of $D$ in the latter group is $N$-torsion after passing to its image in the \emph{relative} Picard group $\Pic_S(E)(S)$ since the cycle class map identifies its identity component with $E$. Thus, $\mathcal{O}_E(ND)$ is the pullback of a bundle on $S$: Writing $e:S\to E$ for the identity section, we have
$e^*\mathcal O_E(ND)\cong\mathcal O_S$, so we conclude that the divisor $\mathrm{lcm}(12,N)D$ is absolutely principal on $E$, and that $N$ alone is sufficient if $D$ is disjoint from zero. To kill the lift ambiguity in $\mathcal{O}(S)^\times$, we observe that each projector $[a]_*-a^2$ for integers $a\equiv 1\pmod{N}$ annihilates this group but acts on $D$ by the scalar $1-a^2$. The gcd of these numbers as $a$ varies is the other term in $d_N$.
\end{proof}

The uniqueness statement gives standard compatibilities between the various $\theta_D$ like multiplicativity in $D$, transformation under isogenies, etc.

\begin{remark} \label{rem:12}
     The disjoint-from-zero bound $d_N \cdot N$ is always an improvement on the $12N^2$ of \cite{BPPS}: one may compute that it is always either $N^2$, $2N^2$, or $4N^2$. We remark that if $D$ is not disjoint from zero, $e^*\mathcal O_E(0)\simeq\omega^{-1}$ for $\omega$ the Hodge bundle on $S$; the $12$th power of the Hodge bundle is famously trivial on the moduli stack of elliptic curves, so we need to multiply by $\mathrm{lcm}(12,N)$ rather than just $N$. Then for these theta functions, the overlapping-zero denominator may sometimes be $24N^2$, so this is not always an improvement. We are somewhat sub-optimal in the remainder of this article and even use $12N d_N$ which may even be $48N^2$; presumably these bounds at $2$ and $3$ could be further improved by closer examination.
\end{remark}

For nonzero $\mathbf{u},\mathbf{v}\in E[N](S)$, define the Siegel unit
\[
G_{\mathbf{u},\mathbf{v}}
:=
e^*\theta_{[\mathbf{v}]-[\mathbf{u}]}\in\mathcal O(S)^\times
\]
where $e$ is the zero section. Then one may compute $G_{\mathbf{u},\mathbf{v}}= 
(\frac{g_{\mathbf u}}{g_{\mathbf v}})^{Nd_N}$ for the classical complex-analytic Siegel unit formula (e.g. as written in \cite[Definition 1.1.1]{BPPS}). We define now the homogeneous formula, for pairwise distinct $\mathbf{a},\mathbf{b},\mathbf{c}$,
\[
\Psi(\mathbf{a}:\mathbf{b}:\mathbf{c}):=\{G_{\mathbf{a-b},\mathbf{b-c}},G_{\mathbf{c-a},\mathbf{b-c}}\}.
\]
Morally, this is
\[(Nd_N)^2(
\{g_{\mathbf a-\mathbf b},
g_{\mathbf c-\mathbf a}\}
+
\{g_{\mathbf c-\mathbf a},
g_{\mathbf b-\mathbf c}\}
+
\{g_{\mathbf b-\mathbf c},
g_{\mathbf a-\mathbf b}\})\]
though the indicated Siegel units are not actually units. However, the construction in \cite{BPPS} shows that they are units once raised to the power of $12N^2$ (so $12^2 N^4$ for the whole expression).

\begin{lemma} \label{lem:321}
For any
$\mathbf a,\mathbf b,\mathbf c,\mathbf x\in E[N](S)$, we have
\[
Nd_N
(
\Psi(\mathbf a:\mathbf b:\mathbf c)
-\Psi(\mathbf a:\mathbf b:\mathbf x)
-\Psi(\mathbf a:\mathbf x:\mathbf c)
-\Psi(\mathbf x:\mathbf b:\mathbf c))
=0.
\]
Here the terms involving a non-pairwise-distinct triple are interpreted as zero.
\end{lemma}
\begin{proof}
    For
    $\mathbf{a},\mathbf{b},\mathbf{c},\mathbf{x}$ all nonzero, let
    \[
    \Xi_{\mathbf{x}}(\mathbf{a},\mathbf{b},\mathbf{c})
    :=
    \{\theta_{[\mathbf{a}]-[\mathbf{x}}],
      \theta_{[\mathbf{b}]-[\mathbf{x}}],
      \theta_{[\mathbf{c}]-[\mathbf{x}}]\}
    \in H^3(U,\Z(3)).
    \]
    Then we may compute, exactly as in \cite[Lemma 3.2.1]{BPPS},
    \[
    (\pi|_{E[N]})_*
    \partial\Xi_{\mathbf{x}}(\mathbf{a},\mathbf{b},\mathbf{c})
    =
    Nd_N(
     \Psi(\mathbf{a}:\mathbf{b}:\mathbf{c})
    -\Psi(\mathbf{a}:\mathbf{b}:\mathbf{x})
    -\Psi(\mathbf{a}:\mathbf{x}:\mathbf{c})
    -\Psi(\mathbf{x}:\mathbf{b}:\mathbf{c})
    )
    \]
    and the result then follows by Theorem \ref{thm:naive}. In general, notice that the identity only involves ``homogeneous'' terms unaffected by the translation $\mathbf{a},\mathbf{b},\mathbf{c},\mathbf{x}\mapsto \mathbf{a}+\mathbf{t},\mathbf{b}+\mathbf{t},\mathbf{c}+\mathbf{t},\mathbf{x}+\mathbf{t}$, so we may simply pick such a $\mathbf{t}$ and apply the same argument for the translated sections.
\end{proof}

Then by Goncharov's finite-group trick applied to $\bigwedge^3 \Z[E[N](S)]$, exactly as in \cite[Proposition 3.2.3]{BPPS}, we deduce that for \emph{any} fixed $\mathbf{a},\mathbf{b},\mathbf{c}\in E[N](S)$, we have
\[
2\sum_{\mathbf{x}\in E[N]}\Psi(\mathbf{a}:\mathbf{b}:\mathbf{x})
=
2\sum_{\mathbf{x}\in E[N]}\Psi(\mathbf{a}:\mathbf{x}:\mathbf{c})
=
2\sum_{\mathbf{x}\in E[N]}\Psi(\mathbf{x}:\mathbf{b}:\mathbf{c})
=0.
\]
from which, summing the preceding lemma over $\mathbf{x}\in E[N](S)$ as in loc. cit., we deduce:
\begin{theorem} \label{thm:bpps}
Let $\mathbf{a},\mathbf{b},\mathbf{c}\in E[N](S)$ be pairwise distinct. Then
\[
2d_N N^3\,
\Psi(\mathbf{a}:\mathbf{b}:\mathbf{c})=0\in H^2(S,\Z(2)).
\]
\end{theorem}

We can also rephrase this in the more familiar form in terms of Siegel units (or rather their powers yielding actual units on the base). In the below statement, we lose some of the integrality of the previous statement; there are possible slight improvements at the primes $2$ and $3$ with some additional casework (which we ignore).

\begin{corollary}
Let
$\mathbf{u},\mathbf{v},\mathbf{w}\in E[N](S)-\{0\}$ satisfy $\mathbf{u}+\mathbf{v}+\mathbf{w}=0$.
Then
\[
2N^3d_N
\left(
 \{g_{\mathbf{u}}^{12N^2},g_{\mathbf{v}}^{12N^2}\}
+\{g_{\mathbf{v}}^{12N^2},g_{\mathbf{w}}^{12N^2}\}
+\{g_{\mathbf{w}}^{12N^2},g_{\mathbf{u}}^{12N^2}\}
\right)
=0 \in H^2(S,\Z(2)).
\]
\end{corollary}

As in loc. cit., this leads to modular symbols. Let $\mathrm{St}(\Q^n)$ be the usual Steinberg module of $\Q^n$ coming from the top reduced homology of the Tits building of $\GL_n(\Q)$; it is generated by apartment classes $[v_1,\ldots, v_n]$ coming from unimodular bases of $\Z^n$. Let $e_1,\ldots, e_n$ be the standard basis of $\Q^n$.
\begin{theorem} \label{thm:modnaive}
    Set $\mathfrak{g}_{\mathbf{x}}= g_{\mathbf{x}}^{12N^2}$. Then there is a $\GL_n(\Z)$-equivariant map 
    \[
    \mathrm{St}(\Q^n) \to \Hom_{\mathrm{Set}}(\Hom(\Z^n,E[N](S)), H^n(S,\Z(n)))
    \]
    sending 
    \[
        [v_1,\ldots, v_n]\mapsto (\phi\mapsto M \{ \mathfrak{g}_{\phi(v_1)}, \ldots, \mathfrak{g}_{\phi(v_n)}\})
    \]
    where we take the convention $\mathfrak{g}_{0}=1$ and $M=2N^3d_N$. Here, the $\GL_n(Z)$-action on the target is $\mu \mapsto (\phi \mapsto \mu(\phi \circ \gamma))$ for the usual left action of $\gamma\in \GL_n(\Z)$ on $\Z^n$. 
\end{theorem}
\begin{proof}
    The equivariance is formal assuming well-definedness. Using the unimodular Bykovskii presentation \cite{Byk} of the Steinberg module $\mathrm{St}(\Q^n)$, it suffices to check that the image under the above-defined map of
    \[
    [e_1,\ldots, e_n]- [e_1+e_2,e_2,\ldots, e_n] + [e_1+e_2,e_1,e_3,\ldots, e_n]
    \]
    vanishes. This amounts to relations of the form 
    \[
    M (\{ \mathfrak{g}_{\mathbf{x}_1}, \mathfrak{g}_{\mathbf{x}_2}\} - \{ \mathfrak{g}_{\mathbf{x}_1+\mathbf{x}_2}, \mathfrak{g}_{\mathbf{x}_2}\} + \{ \mathfrak{g}_{\mathbf{x}_1+\mathbf{x}_2}, \mathfrak{g}_{\mathbf{x}_1}\}) \smile \{\mathfrak{g}_{\mathbf{x}_3}, \ldots, \mathfrak{g}_{\mathbf{x}_n}\}.
    \]
    Note that $\mathfrak{g}_{-\mathbf{x}} = \mathfrak{g}_{\mathbf{x}}$ for any $\mathbf{x}$. When any of $\mathbf{x}_1,\ldots, \mathbf{x}_n, \mathbf{x}_1+\mathbf{x}_2$ are zero, then the relation degenerates into a triviality since the Steinberg symbol is alternating after multiplication by $2$. Otherwise, 
    \[
    M (\{ \mathfrak{g}_{\mathbf{x}_1}, \mathfrak{g}_{\mathbf{x}_2}\} - \{ \mathfrak{g}_{\mathbf{x}_1+\mathbf{x}_2}, \mathfrak{g}_{\mathbf{x}_2}\} + \{ \mathfrak{g}_{\mathbf{x}_1+\mathbf{x}_2}, \mathfrak{g}_{\mathbf{x}_1}\})  = M (\{ \mathfrak{g}_{\mathbf{x}_1}, \mathfrak{g}_{\mathbf{x}_2}\} +\{\mathfrak{g}_{\mathbf{x}_2}, \mathfrak{g}_{-\mathbf{x}_1-\mathbf{x}_2} \} + \{ \mathfrak{g}_{-\mathbf{x}_1-\mathbf{x}_2}, \mathfrak{g}_{\mathbf{x}_1}\})  
    \]
    is precisely a Manin relation of the form we established.
\end{proof}

\subsection{Intrinsic formulation}

As we have phrased it thus far, everything requires only the naive higher Chow group formulation of section 1, besides the parenthetical remark on cdh-sheafified analogues for $S$ not regular. Using the language of $\DM$ of global quotient stacks available from section $2$, we may obtain somewhat finer structural results with motivic coefficients. Along the way, we categorify the assignment of theta functions to degree-zero torsion divisors into a statement on motivic sheaves. 

\begin{remark}In this article, we restrict our attention to understanding integrality issues in the results of \cite{BPPS} for elliptic schemes over regular bases and untwisted motivic coefficients. We expect it may be profitable to apply the formalism below also to (e.g.) cup products of elliptic units, Beilinson's Eisenstein symbols, or even higher abelian polylogarithm classes, which we expect may be shown to satisfy analogous relations.\end{remark}

Let now $\pi:\mathcal{E}\to \mathcal{M}:=\mathcal{M}(\Gamma)$ be the universal elliptic scheme over a smooth global quotient stack of level $\Gamma\le \SL_2(\Z)$: for example, we may take $\mathcal{M}$ to be the moduli stack of all elliptic curves over $\Z$ (which is a global quotient stack via the Weierstrass presentation), or add arbitrary arithmetic level structure so long as we take care to remove non-smooth loci at the ramified primes.

Fix an integer $N>1$ invertible on the base, let $T:=\mathcal{E}[N]$ (a finite \'{e}tale group scheme), $i:T\to \mathcal{E}$ be the closed immersion, and write $j:\mathcal{U}=\mathcal{E}-\mathcal{E}[N]\hookrightarrow \mathcal{E}$ for the complement. Similarly, we write $T_0$ for $\mathcal{E}[N]-\{0\}$, $i_0$ for the inclusion, and $j_0: \mathcal{U}_0\to \mathcal{E}$ for its complement.

The theta function assignment of Proposition \ref{prop:integral-theta} we now upgrade to morphisms of motivic sheaves. Define the ``group ring of $\mathcal{E}[N]_U$'' object
\[
    \DM(\mathcal{M}) \ni P :=[T]_{\mathcal{M}}:= (\pi|_T)_*1_T
\]
to be the motivic sheaf corresponding to the finite flat group scheme $T$. We have a map $(\pi_T)_*:[T] \to 1_{\mathcal{M}}$ in $\DM(\mathcal{M})$ coming from the multiplication-by-$N$ map $[N]:T\to \mathcal{M}$, inducing on pullback an ``augmentation map'' we denote $\varepsilon: P \to 1_U$. We define the ``augmentation ideal'' $I$ by the fiber sequence
\[
    I\to P\xrightarrow{\varepsilon} 1_{\mathcal{M}}
\]
We have a group difference map $\delta: P\otimes P \to P$, coming from the difference map $T\times_{\mathcal{M}} T \to T$ (i.e. the group law precomposed with $(1, \mathrm{inverse})$), as well as a divisor difference map $d: P\otimes P\to P$ coming from the difference of the maps induced by the two projections $T^2\to T$. From the universal property of the cone, we may see that the latter factors through a map we also denote $d: P\otimes P \to I$ (``difference of two torsion sections is zero''). 

When $T$ is a constant group scheme over $\mathcal{M}$, all these above constructions reduce to linear algebra; e.g. $P= \bigsqcup_{\mathbf{x} \in \pi_0(T)} 1_{\mathcal{M}}$, etc. In such a case, associated to an $N$-torsion section $\mathbf{x}$ we will write $1(\mathbf{x}):1_{\mathcal{M}} \to P$ for the inclusion coming from the pushforward from that connected component. 

We will frequently use the base-change of these constructions to the open complement $U$, as well as to the complement of the zero-avoiding $N$-torsion $U_0:= E - (T-\{0\})$, which we denote by subscripts.

\begin{theorem} \label{thm:motivic-integral-theta}
    There exist unique homotopy classes of morphism in $\DM(U)$, respectively $\DM(U_0)$,
    \[
    \theta: I_U\to 1_U(1)[1], \;\;\; \theta_0: I_{0, U_0} \to 1_{U_0}(1)[1] 
    \]
    such that the restriction of $12\cdot \theta_0$ yields $\theta$, and applying $\Hom_{h\DM(U_0)}(1_{U_0,-})$ to the latter recovers the map of Proposition \ref{prop:integral-theta} from degree-zero $N$-torsion divisors on $T_0$ to units on $U_0$. The analogous map applied to $\theta$, correspondingly, recovers the map from degree-zero $N$-torsion divisors on $T$ to units on $U$, needing the extra factor of $12$, as discussed in Remark \ref{rem:12}.

    The statements for $U$, the non-zero avoiding version, hold even without the assumption that $N$ is invertible/$T$ is \'{e}tale.
\end{theorem}
\begin{proof}
    From the above fiber sequences defining the augmentation objects, we get fiber sequences of mapping spaces
\[
\Hom_{\DM(U)}(I_U, 1_U(1)[1]) \to \Hom_{\DM(U)}(1_U, 1_U(1)[2]) \xrightarrow{\varepsilon_*} \Hom_{\DM(U)}(R, 1_U(1)[2])\]
and analogously for the omitting-zero variants, whose corresponding long exact sequence of homotopy groups (i.e. Hom-groups in $h\DM(U)$) yields
\[
   \Gamma(U, \G_m)\xrightarrow{\pi_T^*}  \Gamma(T_U, \G_m)\to \Hom_{h\DM(U)}(I_U, 1_U(1)[1]) \to \Pic(U) \xrightarrow{\pi_T^*} \Pic(T_U).
\]
For the avoiding-zero motives, we analogously obtain 
\[
   \Gamma(U_0, \G_m)\xrightarrow{\pi_{T_0}^*}  \Gamma((T_0)_{U_0}, \G_m)\to \Hom_{h\DM(U_0)}(I_{U_0}, 1_{U_0}(1)[1]) \to \Pic(U_0) \xrightarrow{\pi_{T_0}^*} \Pic((T_0)_{U_0}).
\]
Thus, to single out a homotopy class of morphism in the triangulated category, it is equivalent to give a line bundle $\mathcal{L}$ on $U$ (resp. $U_0$) along with the rigidification data of specific trivializations 
\[
t:\pi_T^*\mathcal{L} \cong \mathcal{O}_{T_U}
\]
on $T_U$ (resp. on $(T_0)_{U_0}$). Indeed, by the standard \v{C}ech descent formalism for line bundles, such data is \emph{equivalent} to specifying a descent datum
\[
    \Theta \in \Gamma(T_U \times_U T_U, \G_m)
\]
subject to the conditions:
\begin{enumerate}
    \item The restriction to the diagonal $\Delta: T_U \hookrightarrow T_U \times_U T_U$ of $\Theta$ is $1$.
    \item The involution on $T_U\times_U T_U$ swapping the two coordinates sends $\Theta$ to $\Theta^{-1}$.
    \item Writing $\pi_{ab}$ for the projection onto the $a$th and $b$th coordinates of $T_U \times_U T_U \times_U T_U$, we have $\pi_{12}^*\Theta \cdot \pi_{23}^*\Theta = \pi_{13}^*\Theta$.
\end{enumerate}
The line bundle $\mathcal{L}$ is then obtained by gluing the trivial line bundle on $\mathcal{O}_{T_U}$ along this descent datum for the finite \'{e}tale cover $T_U\to U$. (Here, we may instead use fppf descent if we drop the assumption that $T$ is \'{e}tale.)

In fact, this is precisely what is produced by the method of Proposition \ref{prop:integral-theta}: in the weight-one motivic localization sequence for 
\[
T_U \times_U T_U \cong U_{T\times T}:= U\times_{\mathcal{M}} T \times_{\mathcal{M}} T,
\]
inside $\mathcal{E}_{T\times T}:= \mathcal{E}\times_{\mathcal{M}} T \times_{\mathcal{M}} T$, observe that we have a pair of tautological sections 
\[
\mathbf p: T \times_{\mathcal{M}} T \to \mathcal{E}_{T\times T},\;\;\; \mathbf q: T \times_{\mathcal{M}} T \to \mathcal{E}_{T\times T}
\]
given by the graphs of the inclusions $T \times_{\mathcal{M}} T \to \mathcal{E}_{T} \times_{\mathcal{M}} T$ and $T \times_{\mathcal{M}} T \to T\times_{\mathcal{M}} \mathcal{E}_{T}$, respectively. Then the difference $[\mathbf p]-[q]$ is a degree-zero $N$-torsion divisor, hence by the already-used Lieberman's trick, $12Nd_N ([\mathbf p]-[\mathbf q])$ lifts to a unique trace-fixed unit on $U_{T\times T}$. We define $\Theta$ to be precisely this lift.

Along the diagonal, $[\mathbf p]-[\mathbf q]$ restricts to the zero divisor, so the restriction of the Lieberman-projected unit is then $1$, showing property (1). By functoriality, the involution of (2) changes $\Theta$ to the lift of $12Nd_N([\mathbf q]-[\mathbf p])$, which is $\Theta^{-1}$ by uniqueness. Finally, we can run this same construction for the three tautological sections $\mathbf p,\mathbf q,\mathbf r$ of 
\[
    T_U \times_U T_U \times_U T_U \cong U_{T\times T\times T}:= U\times_{\mathcal{M}} T \times_{\mathcal{M}} T \times_{\mathcal{M}}T.
\]
Then formally, the construction for, e.g., $[\mathbf p]-[\mathbf r]$ on this space is the pullback of the $2$-arity construction on coordinates $1$ and $3$, and similarly for the other pairs. Hence under the various identifications, the cocycle identity (3) reduces to Lieberman uniqueness together with the identity of divisors
\[
([\mathbf p]-[\mathbf q]) + ([\mathbf q]-[\mathbf r])= ([\mathbf p]-[\mathbf r]).
\]
All of these arguments work similarly for the zero-restricted analogues under the assumption that $T$ is \'{e}tale, except that we need only lift $Nd_N([\mathbf p]-[\mathbf q])$ instead of $12Nd_N([\mathbf p]-[\mathbf q])$ for the reasons previously discussed. 

We take the resulting maps in the triangulated category $I_U\to 1_U(1)[1]$, $I_{U_0}\to 1_{U_0}(1)[1]$, to be our $\theta$ and $\theta_0$. The resulting maps on motivic cohomology may then be computed via the preceding \v{C}ech descent to send $X-Y$ to $\Theta_{X\times Y}$ for equal-degree cycles $X,Y\subset T$, where this notation means the restriction to $U\times_{\mathcal{M}} X \times_{\mathcal{M}} Y \subset U_{T\times T}$, which recovers precisely the construction of Proposition \ref{prop:integral-theta} in the case of full level where $X$ and $Y$ are sections.
\end{proof}

We therefore obtain symbol maps, for any $k\ge 1$,
\[
\theta^{k}: I_U^{\otimes k} \to 1_U(k)[k], \;\;\;\theta_0^{k}: I_{0,U_0}^{\otimes k} \to 1_{U_0}(k)[k]
\]
by taking the $k$th tensor powers of $\theta$ and $\theta_0$ respectively inside $\DM(U)$, respectively $\DM(U_0)$; these maps are $S_k$-equivariant for the permutation action on the source and the sign character on the target.

\begin{remark}
If $T$ is constant, on cohomology the precomposition with the morphism from the unit 
\[
(1(\mathbf{u})-1(\mathbf{v})) \otimes (1(\mathbf{x})-1(\mathbf{y})) :1_{\mathcal{M}} \to I_{\mathcal{M}}^{\otimes 2}
\]
induces the previously-considered Beilinson--Kato element $\{G_{\mathbf{u},\mathbf{v}}, G_{\mathbf{x},\mathbf{y}}\}\in H^2(\mathcal{M},\Z(2))$ for any nonzero sections $\mathbf{u,v,x,y}\in \pi_0(T)$. Our construction's new generality appears when $T$ is not constant, where one may generally only find morphisms from the unit $1_{\mathcal{M}}$ to $I_{\mathcal{M}}^{\otimes 2}$ corresponding to ``glued together combinations'' of torsion sections, yielding lower-level Beilinson--Kato elements.
\end{remark}

Meanwhile, the elements $\Xi$ of the previous section came from considering
\[
\sum_{\mathbf{x}\in T} ([\mathbf{a}]-[\mathbf{x}]) \otimes ([\mathbf{b}]-[\mathbf{x}]) \otimes ([\mathbf{c}]-[\mathbf{x}]) \in H^0(U, I^{\otimes 3})
\]
under the assumption that $T$ was a constant group scheme; Goncharov's summation trick was then applied after taking residue and trace of this map; the pushforwards and residues of these elements were built out of the homogeneous elements $\Psi(\mathbf{a}:\mathbf{b}:\mathbf{c})$.

Here, we instead construct analogues of all of these at the motivic sheaf level: we have a morphism of sheaves
\[
    \Psi: P^{\otimes 3} \to 1_{\mathcal{M}}(2)[2]
\]
as follows: since $T$ is \'{e}tale, we may decompose $T^3$ into the disjoint union of the configuration space over $\mathcal{M}$ of $3$ distinct points on $T$ and the fat diagonal (where points coincide); this corresponds to a decomposition of the motivic sheaves
\[
P^{\otimes 3} = C_3 \sqcup F_3
\]
and in particular we obtain a projection map $P^{\otimes 3}\to C_3$ coming from pullback of the inclusion of that factor. Now let $\pi_{12}: P^{\otimes 3} \to P^{\otimes 2}$ come from the projection of $T^3$ onto the first two coordinates, and so on for the other pairs of indices. 

On $C_3$, one may check via the fiber sequence that the composition of its pushforward inclusion on $P^{\otimes 3}$ with
\[
(\delta\circ \pi_{12})\otimes (\delta\circ \pi_{23}) \otimes (\delta\circ \pi_{31}): P^{\otimes 3}\to P^{\otimes 3}
\]
factors through $P^{\otimes 3}_0$ (i.e. ``differences of distinct torsion sections avoid zero''), and we call this map
\[
   C(\delta):  C_3 \to P^{\otimes 3}_0.
\]
Write correspondingly the Siegel unit specializations
\[
G: P_0^{\otimes 2} \xrightarrow{d} I_0 \xrightarrow{e^*\theta_0} 1_{\mathcal{M}}(1)[1], 
\]
and so on. We may now define $\Psi$ to be the composite
\[
P^{\otimes 3} \xrightarrow{\mathrm{project}} C_3 \xrightarrow{C(\delta)} P^{\otimes 3}_0 \xrightarrow{G_{12} \smile G_{32}} 1_{\mathcal{M}}(2)[2] 
\]
where $G_{ij}:= G \circ \pi_{ij}$, and the cup product 
\[
G_{12} \smile G_{32}: P_0^{\otimes 3} \to 1_{\mathcal{M}}(2)[2],
\]
is constructed using the coproduct as the composite
\[
P_0^{\otimes 3} \xrightarrow{\Delta^{\otimes 3}} P_0^{\otimes 3} \otimes P_0^{\otimes 3} \xrightarrow{G_{12} \otimes G_{32}} 1_{\mathcal{M}}(1)[1] \otimes  1_{\mathcal{M}}(1)[1] \cong 1_{\mathcal{M}}(2)[2].
\]
In the full-level $N$ setting, the precomposition of this map with 
\[
1(\mathbf{u})\otimes 1(\mathbf{v}) \otimes 1(\mathbf{w}): 1_{\mathcal{M}} \to P^{\otimes 3}
\]
yields on motivic cohomology precisely the element $\Psi(\mathbf{u}:\mathbf{v}:\mathbf{w})$ defined in the previous section. (Note that the projection of $P^{\otimes 3}$ onto $C_3$ exactly is the ``extension by zero by fiat''.)

We define now
\[
\Xi: P^{\otimes 4}_U \to 1_{U}(3)[3]
\]
by the composite
\[
P^{\otimes 4}_U \xrightarrow{\mathrm{project}} C_3 \xrightarrow{D} I^{\otimes 3}_{0,U} \xrightarrow{\theta^{\otimes 3}} 1_U(3)[3].
\]
where $D=(d\circ \pi_{14}) \otimes (d\circ \pi_{24}) \otimes (d\circ \pi_{34})$. With full level, the ``value'' in motivic cohomology of this map at $1(\mathbf{a})\otimes 1(\mathbf{b})\otimes 1(\mathbf{c}) \otimes 1(\mathbf{x})$ is $\Xi_{\mathbf{x}}(\mathbf{a},\mathbf{b}, \mathbf{c})$ in the notation of the previous section. 

By adjunction, we obtain a map in $\DM(\mathcal{M})$
\[
\Xi^\sharp: P^{\otimes 4} \to \pi_*1_U(3)[3].
\]
We write $\mathcal{R}$ in $\DM(\mathcal{M})$
\[
\pi_*1_U(3)[3] \xrightarrow{\partial} P(2)[2] \xrightarrow{\pi_*} 1_{\mathcal{M}}(2)[2]
\]
for the indicated composition of the pushforward by $\pi$ of the residue map in the localization sequence for $1_\mathcal{E}(3)[3]$ on $\mathcal{E}$, with the trace by $\pi|_T:T\to \mathcal{M}$).

The following is the lift of \cite[Lemma 3.2.1]{BPPS} to motivic sheaves:
\begin{proposition}
We have
\[
    \mathcal{R}\circ \Xi^\sharp = Nd_N (\Psi \circ \pi_{123} - \Psi \circ \pi_{124}-\Psi \circ \pi_{143}-\Psi \circ \pi_{423})
\]  
as homotopy classes of maps $P^{\otimes 4}\to 1_{\mathcal{M}}(2)[2]$ in $h\DM(\mathcal{M})$.
\end{proposition}
\begin{proof}
Homotopy classes of maps
\[
P^{\otimes 4}\to 1_{\mathcal M}(2)[2]
\]
are the same, by adjunction of pushforward/pullback under $T^4\to \mathcal{M}$, as classes in
\[
\Hom_{h\DM(T^4)}(1_{T^4},1_{T^4}(2)[2])
\]
and it thus suffices to check that the two sides correspond under adjunction to the same motivic cohomology class in $T^4$: tracing the definition of the adjunction, this motivic cohomology class is none other than this very same construction base changed from $\mathcal{M}$ to $T^4$, applied to tensor product of the four tautological coaugmentations in $P_{T^4}^{\otimes 4}$, $1(\mathbf a)\otimes 1(\mathbf b) \otimes 1(\mathbf c)\otimes 1(\mathbf x)$. But then this calculation then just formally reduces to the elementary one in \cite[Lemma 3.2.1]{BPPS} we cited in Lemma \ref{lem:321}; i.e., it reduces to an equality of formal cup products
\[
    ((\pi\times_{\mathcal{M}} T^4)|_{T_{T^4}})_*
    \partial\Xi_{\mathbf{x}}(\mathbf{a},\mathbf{b},\mathbf{c})
    =
    Nd_N(
     \Psi(\mathbf{a}:\mathbf{b}:\mathbf{c})
    -\Psi(\mathbf{a}:\mathbf{b}:\mathbf{x})
    -\Psi(\mathbf{a}:\mathbf{x}:\mathbf{c})
    -\Psi(\mathbf{x}:\mathbf{b}:\mathbf{c})
    ).
\]
\end{proof}

Write now $\eta:1_{\mathcal{M}}\to P$ for the ``sum over $T$'' coaugmentation map coming from pullback by the structure map of $T$. The ``summing over $\mathbf{x}$'' trick of the previous section now amounts to precomposing the equation of relations in the preceding proposition with $\mathrm{id}^{\otimes 3} \otimes \eta: P^{\otimes 3} \to P^{\otimes 4}$. Then we have the categorified Goncharov trick:

\begin{lemma}
    We have $2 \cdot \Psi \circ (\mathrm{id}^{\otimes 2} \otimes \eta)=0$ as maps $P^{\otimes 2} \to 1_{\mathcal{M}}(2)[2]$ (and similarly, by symmetry, for the other orderings of $\mathrm{id},\mathrm{id},\eta$ in the three tensor factors).
\end{lemma}
\begin{proof}
    Let, as usual, $\pi_{12}:T^3\longrightarrow T^2$ be the projection. Using the same adjunction trick as the previous proof to identify our map with an element of $H^2(T^2, \Z(2))$, our element 
    \[
        \Sigma:=\Psi\circ(\mathrm{id}^{\otimes2}\otimes\eta)
    \]
    corresponds to
    \[
        (\pi_{12})_*\Psi(\mathbf a:\mathbf b:\mathbf x).
    \]
    where here this the map $\Psi$ precomposed with $1(\mathbf{a})\otimes 1(\mathbf{b}) \otimes 1(\mathbf{x})$ for $\mathbf{a},\mathbf{b},\mathbf{x}$ the tautological coaugmentations of $P^{\otimes 3}_{T^3}$.

    For each pair $\mathbf a,\mathbf b$, we have
    \[
    \Psi(\mathbf a:\mathbf b:\mathbf x)
    =
    \Psi(-\mathbf b:-\mathbf a:\mathbf x-\mathbf a-\mathbf b).
    \]
    The automorphism $T^3\to T^3$ given by
    \[
    (\mathbf a,\mathbf b,\mathbf x)
    \mapsto
    (\mathbf a,\mathbf b,\mathbf x-\mathbf a-\mathbf b)
    \]
    is an automorphism over $T_2$, hence it does not change the trace along
    $q$, and the change of variable $\mathbf y:=\mathbf x-\mathbf a-\mathbf b$ yields
    \[
    \Sigma
    =
    (\pi_{12})_*\Psi(-\mathbf b:-\mathbf a:\mathbf y).
    \]
    Since $\Psi$ is alternating in its three arguments,
    \[
    (\pi_{12})_*\Psi(-\mathbf b:-\mathbf a:\mathbf y)
    =
    -(\pi_{12})_*\Psi(-\mathbf a:-\mathbf b:\mathbf y).
    \]
    By invariance under simultaneous inversion,
    \[
    \Psi(-\mathbf a:-\mathbf b:\mathbf y)
    =
    \Psi(\mathbf a:\mathbf b:-\mathbf y).
    \]
    Finally, $\mathbf y\mapsto-\mathbf y$ is another automorphism of the fibers of $q$, so
    \[
    \Sigma
    =
    -\Sigma, \mapsto 2\Sigma =0.
    \]
\end{proof}

We therefore finally deduce the categorified version of the rank-$2$ Bykovskii relation:
\begin{theorem} \label{thm:sheafrel}
    We have $2N^3d_N \Psi =0$ as maps $P^{\otimes 3} \to 1_{\mathcal{M}}(2)[2]$. 
\end{theorem}
\begin{proof}
    By Theorem \ref{thm:main}, $R\circ \Xi^\sharp \circ (\mathrm{id}^{\otimes 3} \otimes \eta)=0$, and the terms involving $\pi_{124},\pi_{143},\pi_{423}$ on the other side in the preceding proposition vanish by the Goncharov trick. The conclusion follows.
\end{proof}

At full level, evaluating the previous theorem in motivic cohomology at $1(\mathbf{a})\otimes 1(\mathbf{b})\otimes 1(\mathbf{c})$ yields the previous section's identity
\[
N^3d_N\Psi(\mathbf{a}:\mathbf{b}:\mathbf{c})=0
\]
which may be used to construct modular symbols valued in distributions over Beilinson--Kato elements at full level. More generally, even if $\mathcal{M}$ has coarser than full-level-$N$ structure, one can associate coaugmentation maps $1_{\mathcal{M}}\to P$ to the connected components of $T$ over $\mathcal{M}$, and thereby obtain refined identities between Beilinson--Kato elements at this coarser level.

As before, we may package our results into modular symbols. For this, we need a ``pulled-back to the base'' sheafified version of the individual Beilinson--Kato units $\mathfrak{g}_{\mathbf{a}}:=g_{\mathbf{a}}^{12 N^2}$. Indeed, the same adjunction trick we have been using provides this: we have a universal element
\[
    \mathfrak{g}'\in H^1(T-\{0\},\Z(1))
\]
coming from running the standard construction for the bundle $E_T$ (whose specialization produced $\mathfrak{g}_{\mathfrak{a}}$ over a full-level base) and specializing at its tautological $N$-torsion section. We upgrade this to an element $\mathfrak{g}\in H^1(T, \Z(1))$ by simply taking the multiplicative identity on the connected component $\{0\}$ over $\mathcal{M}$. Under our adjunction, this exactly gives a morphism
\[
    \mathfrak{g}: P \to 1_{\mathcal{M}}(1)[1]
\]
which specializes at full level under precomposition with $1(\mathbf{a})\to P$ to $\mathfrak{g}_{\mathbf{a}}$. Then via the same adjunction argument, the identity
\[
    \mathfrak{g} \circ \pi_1 - \mathfrak{g} \circ \pi_2 = G \in \Hom_{h\DM(\mathcal{M})}(P^{\otimes 2}_0, 1_{\mathcal{M}}(1)[1]) 
\]
follows from the same universal identity on $T^2$, computable precisely as in the previous section. Plugging this into the definition of $\Psi$, we may then deduce:
\begin{proposition}
We have an identity of homotopy classes of maps
    \[
    2N^3d_N\Psi = 2N^3d_N(\mathfrak{g} \smile \mathfrak{g}) \circ \delta^{\otimes 2} \circ (\pi_{12} \otimes \pi_{23}+\pi_{23} \otimes \pi_{31}+\pi_{31} \otimes \pi_{12})
    \]
between $P^{\otimes 3}$ and $1_{\mathcal{M}}(2)[2]$, and consequently, both sides equal zero.
\end{proposition}

We may organize this into a modular symbol as follows: the maps $\mathfrak{g}^{\otimes n}:P^{\otimes n} \to 1_{\mathcal{M}}(n)[n]$ are precisely the motivic sheaf analogue of the distributions $\boldsymbol{\mu}^n$ in \cite{BPPS}. 

Let $\GL_n(\Z)$ act on $P^{\otimes n}$ via its action on $T^n$ identified as $\underline{\Hom}(\Z^n, T)$ (as functors of points on the \'{e}tale site of $\mathcal{M}$), and $\Z^n$ given its usual left action. Note that when $T$ is the constant group scheme, this is compatible with the action defined before Theorem \ref{thm:modnaive}. 

Then the following is deduced from the previous proposition in the same way as that theorem:

\begin{theorem} \label{thm:mod}
     The assignment 
     \[
     \mathrm{St}(\Q^n) \to \Hom_{h\DM(\mathcal{M})}(P^{\otimes n}, 1_{\mathcal{M}}(n)[n])
     \]
     sending $[e_1,\ldots, e_n]\mapsto 2N^3 d_N\mathfrak{g}^{\otimes n}$ and extended by formal $\GL_n(\Z)$-equivariance gives a well-defined modular symbol, compatible with base change between different $\mathcal{M}$.
\end{theorem}

As usual, at full level, precomposing with suitable tensors of the trivialized coaugmentations $1(\mathbf{x})$ exactly recovers Theorem \ref{thm:modnaive}, and thus also, upon rationalization and base change up the infinite-level modular tower over $\mathbb{C}$, the main result of \cite{BPPS}. In general, at coarser levels we obtain finer integral results from coaugmentations coming from glued-together combinations of torsion sections.

\begin{remark}
We remark that when $\mathcal{E}$ is replaced by an elliptic curve with CM by an imaginary quadratic order $\mathcal{O}$ with fraction field $K$, the exact same methods afford us also modular symbols for $\mathrm{St}(K^n)$ only when $\mathcal{O}$ is a PID and admits a Bykovskii presentation, which is generally not the case (though it is true for $\mathcal{O}=\Z[i],\Z[\omega]$ \cite{KMPW}). More relations are needed to obtain a modular symbol in general, and we expect that variants of our Suslin reciprocity argument should yield them (cf. our joint article \cite{SX} in the function field setting). In particular, this could be applied to the Sharifi conjectures when $n=2$, which may be worth further investigation (cf. \cite{LSSW}).
\end{remark}

{
\printbibliography
}

\end{document}